\documentclass[a4paper, 11pt]{article}
\usepackage{mathrsfs}
\usepackage{mathrsfs}
\usepackage{amsmath,amssymb}
\usepackage{amsfonts}
\usepackage[T1]{fontenc}
\usepackage[latin9]{inputenc}
\usepackage{amsthm}
\usepackage{graphicx}
\usepackage{amsmath}

\usepackage{amsthm}
\usepackage{array}
\usepackage{cases}
\usepackage{enumerate,enumitem,todonotes}
\usepackage{thmtools}
\usepackage{thm-restate}
\usepackage{bbm}

\usepackage{xcolor}
\usepackage{soul} 

\newcommand{\mad}{\mathrm{mad}}
\newcommand{\ch}{\mathrm{ch}}

\makeatletter
\def\th@plain{%
  \itshape 
}
\makeatother

\makeatletter
\renewenvironment{proof}[1][\proofname]{\par
  \pushQED{\qed}%
  \normalfont \topsep6\p@\@plus6\p@\relax
  \trivlist
  \item[\hskip\labelsep
        \bfseries
    #1\@addpunct{.}]\ignorespaces
}{%
  \popQED\endtrivlist\@endpefalse
}
\makeatother

\numberwithin{equation}{section}

\newtheorem{thm}{Theorem}[section]
\newtheorem{cor}[thm]{Corollary}

\newtheorem{lemma}[thm]{Lemma}

\newtheorem{example}[thm]{Example}

\newtheorem{oq}[thm]{Open Question}

\numberwithin{equation}{section}

\usepackage[varg]{txfonts}
\usepackage{graphicx, epsfig, subfigure}
\usepackage{enumerate}
\usepackage[square, numbers, sort&compress]{natbib}
\ifx\pdfoutput\undefined
 \usepackage[dvipdfm,%
  pdfstartview=FitH,
  bookmarks=true,%
  bookmarksnumbered=true,
  bookmarksopen=true,
  plainpages=false,%
  pdfpagelabels,%
  colorlinks=true,
  linkcolor=blue,
  citecolor=blue,%
  urlcolor=black,
  pdfborder=001]{hyperref}
  \AtBeginDvi{}  
\else
 \usepackage[pdftex,%
  pdfstartview=FitH,
  bookmarks=true,%
  bookmarksnumbered=true,
  bookmarksopen=true,
  plainpages=false,%
  pdfpagelabels,%
  colorlinks=true,
  linkcolor=blue,
  citecolor=blue,%
  urlcolor=black,
  pdfborder=001]{hyperref}%
\fi

\usepackage[left]{lineno}
\RequirePackage[normalem]{ulem} 
\RequirePackage{color}\definecolor{RED}{rgb}{1,0,0}\definecolor{BLUE}{rgb}{0,0,1} 
\providecommand{\DIFaddbegin}{} 
\providecommand{\DIFaddend}{} 
\providecommand{\DIFdelbegin}{} 
\providecommand{\DIFdelend}{} 
\providecommand{\DIFaddbeginFL}{} 
\providecommand{\DIFaddendFL}{} 
\providecommand{\DIFdelbeginFL}{} 
\providecommand{\DIFdelendFL}{} 
\newcommand{\DIFscaledelfig}{0.5}
\RequirePackage{settobox} 
\RequirePackage{letltxmacro} 
\newsavebox{\DIFdelgraphicsbox} 
\newlength{\DIFdelgraphicswidth} 
\newlength{\DIFdelgraphicsheight} 
\LetLtxMacro{\DIFOincludegraphics}{\includegraphics} 
\newcommand{\DIFaddincludegraphics}[2][]{{\color{blue}\fbox{\DIFOincludegraphics[#1]{#2}}}} 
\newcommand{\DIFdelincludegraphics}[2][]{
\sbox{\DIFdelgraphicsbox}{\DIFOincludegraphics[#1]{#2}}
\settoboxwidth{\DIFdelgraphicswidth}{\DIFdelgraphicsbox} 
\settoboxtotalheight{\DIFdelgraphicsheight}{\DIFdelgraphicsbox} 
\scalebox{\DIFscaledelfig}{
\parbox[b]{\DIFdelgraphicswidth}{\usebox{\DIFdelgraphicsbox}\\[-\baselineskip] \rule{\DIFdelgraphicswidth}{0em}}\llap{\resizebox{\DIFdelgraphicswidth}{\DIFdelgraphicsheight}{
\setlength{\unitlength}{\DIFdelgraphicswidth}
\begin{picture}(1,1)
\thicklines\linethickness{2pt} 
{\color[rgb]{1,0,0}\put(0,0){\framebox(1,1){}}}
{\color[rgb]{1,0,0}\put(0,0){\line( 1,1){1}}}
{\color[rgb]{1,0,0}\put(0,1){\line(1,-1){1}}}
\end{picture}
}\hspace*{3pt}}} 
} 
\LetLtxMacro{\DIFOaddbegin}{\DIFaddbegin} 
\LetLtxMacro{\DIFOaddend}{\DIFaddend} 
\LetLtxMacro{\DIFOdelbegin}{\DIFdelbegin} 
\LetLtxMacro{\DIFOdelend}{\DIFdelend} 
\DeclareRobustCommand{\DIFaddbegin}{\DIFOaddbegin \let\includegraphics\DIFaddincludegraphics} 
\DeclareRobustCommand{\DIFaddend}{\DIFOaddend \let\includegraphics\DIFOincludegraphics} 
\DeclareRobustCommand{\DIFdelbegin}{\DIFOdelbegin \let\includegraphics\DIFdelincludegraphics} 
\DeclareRobustCommand{\DIFdelend}{\DIFOaddend \let\includegraphics\DIFOincludegraphics} 
\LetLtxMacro{\DIFOaddbeginFL}{\DIFaddbeginFL} 
\LetLtxMacro{\DIFOaddendFL}{\DIFaddendFL} 
\LetLtxMacro{\DIFOdelbeginFL}{\DIFdelbeginFL} 
\LetLtxMacro{\DIFOdelendFL}{\DIFdelendFL} 
\DeclareRobustCommand{\DIFaddbeginFL}{\DIFOaddbeginFL \let\includegraphics\DIFaddincludegraphics} 
\DeclareRobustCommand{\DIFaddendFL}{\DIFOaddendFL \let\includegraphics\DIFOincludegraphics} 
\DeclareRobustCommand{\DIFdelbeginFL}{\DIFOdelbeginFL \let\includegraphics\DIFdelincludegraphics} 
\DeclareRobustCommand{\DIFdelendFL}{\DIFOaddendFL \let\includegraphics\DIFOincludegraphics} 
\begin{document}
\title{\LARGE  
On $(1^2,2^2)$-packing edge-coloring of sparse subcubic graphs 
}
\author{
Xujun Liu\thanks{Department of Applied Mathematics, School of Mathematics and Physics, Xi'an Jiaotong-Liverpool University, Suzhou, Jiangsu Province, 215123, China, \texttt{xujun.liu@xjtlu.edu.cn}; the research of X. Liu was supported by the National Natural Science Foundation of China under grant No.~12401466.}
\and
Jiacheng Yang\thanks{School of Mathematics and Physics, Xi'an Jiaotong-Liverpool University, Suzhou, Jiangsu Province, 215123, China, \texttt{jiacheng.yang24@student.xjtlu.edu.cn}.}
\and
Xin Zhang\thanks{School of Mathematics and Statistics, Xidian University, Xi'an, Shaanxi Province, 710126, China, \texttt{xzhang@xidian.edu.cn}.}
}

\date{}

\maketitle

\begin{abstract}
\baselineskip 0.60cm

For positive integers $\ell$ and $k$, a $(1^\ell, 2^k)$-packing edge-coloring of a graph $G$ is a partition of $E(G)$ into $\ell$ matchings and $k$ induced matchings. A graph is $d$-irregular if it has no adjacent vertices of degree $d$. Yang and Wu proved that every $3$-irregular subcubic graph admits a $(1,2^4)$-packing edge-coloring, which answered an open question of Hocquad, Lajou, and Lu\v zar in the affirmative. In this paper, we prove an analogue result that every $3$-irregular subcubic multigraph is $(1^2,2^2)$-packing edge-colorable. Our result is sharp since there are $3$-irregular subcubic graphs that are not $(1,2^3)$-packing edge-colorable and $(1^2,2)$-packing edge-colorable, respectively.

Hocquad, Lajou, and Lu\v zar conjectured that every subcubic planar graph is $(1^2,2^3)$-packing edge-colorable. Furthermore, they found a subcubic planar graph with girth $3$ that is not $(1^2,2^2)$-packing edge-colorable. For every fixed integer $k \ge 3$, we found graphs with girth $k$ that are not $(1^2,2)$- and not $(1,2^3)$-packing edge-colorable. It is natural to consider the question "what is the minimum positive integer $g$ such that every subcubic planar graph with girth at least $g$ is $(1^2,2^2)$-packing edge-colorable?". We prove $g$ is finite and in fact $g \le 20$. We also provide an example showing $g \ge 6$.

\vspace{3mm}\noindent \emph{Keywords}: subcubic graphs; packing edge-colorings; irregular graphs; planar graphs; girth.
\end{abstract}

\baselineskip 0.60cm

\section{Introduction}

An {\em $S$-packing edge-coloring} of a graph $G$, where $S = (s_1, \ldots, s_k)$ is a non-decreasing sequence of integers, is a partition of the edge set $E(G)$ into $E_1, \ldots, E_k$ such that for every pair of distinct edges $e_1,e_2 \in E_i$, $1 \le i \le k$, the distance between $e_1$ and $e_2$ (the vertex distance in the line graph between their corresponding vertices) is at least $s_i + 1$. The notion of $S$-packing edge-coloring was first generalized by Gastineau and Togni~\cite{GT1} from its vertex counterpart, which has also received a significant amount of attention from many researchers (e.g.~\cite{BKL1, BF1, BKRW1, FKL1, GT2}). To simplify the notation, we denote repetitions of the same numbers in $S$ using exponents and omit the word ``packing'' in this paper. For example, $(1,1,2,2,2)$-packing edge-coloring is denoted by $(1^2,2^3)$-edge-coloring, and furthermore, we say there are two $1$-colors and three $2$-colors. More generally, when there are multiple colors of the same type, we use $1_a, 1_b, \ldots$ to denote the $1$-colors and use $2_a, 2_b, \ldots$ to denote the $2$-colors. If there is only one 1-color or 2-color, we simply denote it by 1 or 2, respectively.

We focus on the case where each $s_i \in \{1,2\}$, $1 \le i \le k$. Using the $S$-packing edge-coloring notation, a $(1^{\ell})$-edge-coloring is a proper $\ell$ edge-coloring. By Vizing's theorem~\cite{V1} on edge coloring, the minimum $\ell$ such that a simple graph $G$ has a $(1^{\ell})$-edge-coloring is either $\Delta(G)$ or $\Delta(G) + 1$, where we denote the former class by Class I and the latter by Class II. Furthermore, a $(2^k)$-edge-coloring is a strong edge-coloring using $k$ colors. The strong edge-coloring was first introduced by Fouquet and Jolivet~\cite{FJ1}. Erd\H{o}s and Ne\v set\v ril~\cite{EN1} conjectured that the minimum $k$ such that $G$ has a $(2^k)$-edge-coloring is $\frac{5}{4} \Delta(G)^2$ when $\Delta(G)$ is even and $\frac{5}{4} \Delta(G)^2 - \frac{1}{2} \Delta(G) + \frac{1}{4}$ when $\Delta(G)$ is odd. Strong edge-coloring for graphs with small maximum degree was studied in~\cite{A1, HQT1, HSY1, KLRSWY1}. For $\Delta(G)$ large, progresses were made in~\cite{CKKR1, FKS1, LMSS1, MR1} and the current best upper bound ($1.772 \Delta(G)^2$) was proved by Hurley, de Joannis de Verclos, and Kang~\cite{HJK1} by probabilistic method.

Gastineau and Togni~\cite{GT1} studied the $S$-packing edge-coloring of subcubic graphs with a prescribed number of $1$'s in the sequence. They asked an open question that ``is it true that every subcubic graph has a $(1,2^7)$-packing edge-coloring?'' and conjectured every subcubic graph has a $(1^2,2^4)$-packing edge-coloring. Hocquad, Lajou, and Lu\v zar~\cite{HLL2} showed that every subcubic graph has a $(1,2^8)$-packing edge-coloring and a $(1^2,2^5)$-packing edge-coloring. They also conjectured that every subcubic graph has a $(1,2^7)$-packing edge-coloring and a $(1^2,2^4)$-packing edge-coloring. Liu, Santana, and Short~\cite{LSS1} answered the first conjecture of Hocquard Lajou, and Lu\v zar in the positive, and Liu and Yu~\cite{LY1} solved the second conjecture. Furthermore, Hocquard, Lajou, and Lu\v zar~\cite{HLL2} asked the following open question.

\begin{oq}[Hocquad, Lajou, and Lu\v zar~\cite{HLL2}]\label{question1}
Is it true that every subcubic bipartite graph, in which each edge has weight at most $5$ (i.e., $3$-irregular), admits a $(1,2^4)$-packing edge-coloring? 
\end{oq}

Question~\ref{question1} was answered in the affirmative by Yang and Wu~\cite{YW1}. In fact, they showed every $3$-irregular subcubic graph (not necessarily bipartite) admits a $(1,2^4)$-packing edge-coloring. In this paper, we first show an analogue result for $(1^2,2^2)$-packing edge-coloring of $3$-irregular subcubic multigraph.

\begin{thm}\label{theorem-1}
Every $3$-irregular subcubic multigraph has a $(1^2,2^2)$-packing edge-coloring.
\end{thm}

Our result is sharp in the sense that there exists $3$-irregular subcubic graphs that are not $(1,2^3)$-packing edge-colorable and $(1^2,2)$-packing edge-colorable, respectively.

\begin{example}
Let $G_1$ be the graph shown in Figure~\ref{example1}. We first note that $G_1$ is a $3$-irregular subcubic graph. We show that $G_1$ has no $(1,2^3)$-packing edge-coloring and no $(1^2,2)$-packing edge-coloring. 
\end{example}

\begin{proof}
Note that $G_1$ has six vertices and thus a maximum matching can have size at most three. Furthermore, $G_1$ only has two maximum matchings, i.e., $M_1 = \{uu_2, u_1u_5, u_3u_4\}$ and $M_2 = \{uu_5, u_1u_2, u_3u_4\}$. However, they are not disjoint matchings and thus a disjoint union of two matchings can contain at most five edges. Since every pair of edges in $G_1$ has distance at most two, each $2$-color can be used at most once. Since $G_1$ has seven edges, it has no $(1,2^3)$-packing edge-coloring and no $(1^2,2)$-packing edge-coloring.  
\end{proof}

\begin{figure}
\vspace{-5mm}
\begin{center}
\hspace{-8mm}
\includegraphics[scale=0.7]{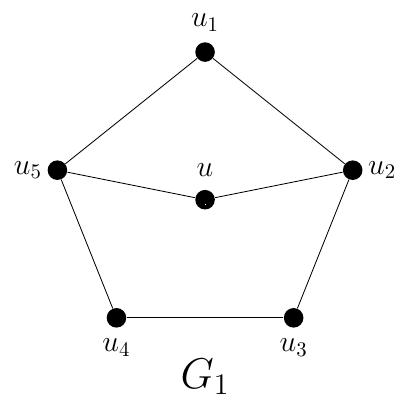} 
\includegraphics[scale=0.6]{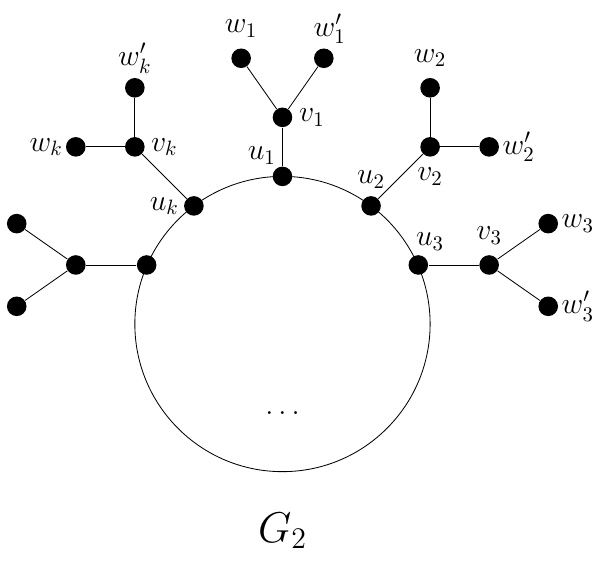} 
\includegraphics[scale=0.45]{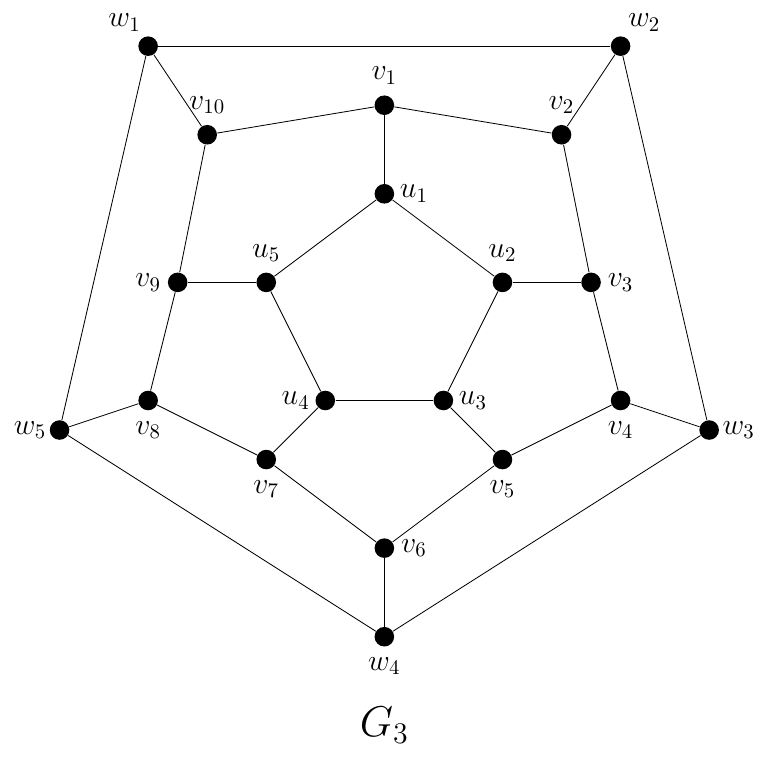}
\caption{Examples.}\label{example1}
\end{center}
\vspace{-8mm}
\end{figure}

The maximum average degree ($\mad(G)$) of a graph $G$, defined to be $\max\{\frac{2|E(H)|}{|V(H)|} \text{ }| \text{ }H \subseteq G\}$, is a different (to $d$-irregular graphs) way to measure the sparseness of a graph. We prove that if a subcubic graph has $\mad(G) < \frac{20}{9}$, then it has a $(1^2,2^2)$-packing edge-coloring.

\begin{thm}\label{theorem-2}
Every subcubic graph $G$ with $\mad(G) < \frac{20}{9}$ has a $(1^2,2^2)$-packing edge-coloring.
\end{thm}

Hocquad, Lajou, and Lu\v zar~\cite{HLL2} conjectured that every subcubic planar graph has a $(1^2, 2^3)$-packing edge-coloring. Furthermore, they found a subcubic planar graph with girth $3$ that is not $(1^2,2^2)$-packing edge-colorable. We first provide two examples showing for every fixed positive integer $k$, there is a subcubic planar graph $H$ with girth at least $k$ such that $H$ has no $(1^2,2)$- and no $(1,2^3)$-packing edge-coloring. 

\begin{example}
Let $G_2$ be the graph in Figure~\ref{example1}, with girth equal to $k$, where $k$ is a fixed positive integer greater than or equal to $3$. We show it has no $(1^2,2)$- and no $(1,2^3)$-packing edge-coloring.      
\end{example}

\begin{proof}
We first show $G_2$ has no $(1^2,2)$-packing edge-coloring. By symmetry, the edges $u_ku_1, u_1u_2, u_2u_3$ can be colored with $1_a,1_b,1_a$ or $2, 1_a, 1_b$ or $1_a, 2, 1_a$ or $1_a, 2, 1_b$. We obtain a contradiction in each case. In the first case, $u_1v_1$ and $u_2v_2$ are both colored with $2$. In the second case, $u_2v_2$ and $u_1u_k$ are both colored with $2$. In the third and fourth cases, $u_1u_2$ is colored with $2$ and one edge incident with $v_1$ must be colored with $2$.

We next show $G_2$ has no $(1,2^3)$-packing edge-coloring. If $u_1u_2$ is colored with $1$, then $u_1v_1, u_1u_k, u_2v_2, u_2u_3$ must receive distinct $2$-colors. This is a contradiction and thus $u_1u_2$ must receive a $2$-color, say $2_a$. Then one of $u_2v_2$ and $u_2u_3$ must be colored with $1$. Suppose not, say $u_2v_2, u_2u_3$ are colored with $2_b, 2_c$. Then $u_1v_1, u_1u_k$ must both use color $1$. This is a contradiction. Say $u_2v_2$ is colored with $1$ (the proof for $u_2u_3$ is colored with $1$ is very similar) and $u_2u_3$ is colored with $2_b$. However, $v_2w_2$ and $v_2w_2'$ must be both colored with $2_c$. This is again a contradiction.
\end{proof}

We are interested in the question "what is the minimum positive integer $g$ such that every subcubic planar graph with girth at least $g$ is $(1^2,2^2)$-packing edge-colorable?". By Euler's formula, every planar graph with girth at least $\hat{g}$ has maximum average degree less than $\frac{2\hat{g}}{\hat{g}-2}$. We obtain the following corollary of Theorem~\ref{theorem-2} for subcubic planar graphs, which implies $g$ is finite and in fact $g \le 20$.

\begin{cor}\label{maincor}
Every subcubic planar graph with girth at least $20$ has a $(1^2,2^2)$-packing edge-coloring.    
\end{cor}

We provide the following examples, which shows $g \ge 6$.

\begin{example}
Let $G_3$ be the graph in Figure~\ref{example1} with girth $5$. We show it has no $(1^2,2^2)$-packing edge-coloring.   
\end{example}

\begin{proof}
First note that each $2$-color can be used at most once on the $5$-cycle $C = u_1u_2u_3u_4u_5u_1$. If both $2_a, 2_b$ are used on $C$, then they cannot be adjacent, since otherwise, say $u_1u_2, u_2u_3$ are colored with $2_a, 2_b$, each of $v_2v_3, v_3v_4, u_2v_3$ can only use $1_a$ or $1_b$. This is a contradiction. Therefore, we may assume $u_1u_2, u_3u_4$ are colored with $2_a, 2_b$. However, each of $u_1u_5, u_4u_5, u_5v_9$ can only use $1_a$ or $1_b$. This is again a contradiction. 

Hence, at most one $2$-color can be used on $C$. Since $C$ is a $5$-cycle, it must use exactly one $2$-color, say $2_a$ at $u_1u_2$, and $u_2u_3, u_3u_4, u_4u_5, u_5u_1$ is colored with $1_a, 1_b, 1_a, 1_b$. Then $u_3v_5, u_4v_7$ must be colored with $2_b, 2_a$. However, each of $v_5v_6, v_6v_7, v_6w_4$ can only use $1_a$ or $1_b$. This is a contradiction. 
\end{proof}

We say a vertex $u$ {\em sees} a color $x$ if there is an edge colored with $x$ and $u$ is one of its endpoint. We prove Theorem~\ref{theorem-1} in Section 2 and Theorem~\ref{theorem-2} in Section 3. 

\section{Proof for Theorem~\ref{theorem-1}}

We first show how we deal with multi-edges after we proved the statement ``Every $3$-irregular subcubic simple graph has a $(1^2, 2^2)$-packing edge-coloring.''. Suppose to the contrary that there are $3$-irregular subcubic multigraphs that are not $(1^2,2^2)$-packing edge-colorable. We take a graph $G$ such that $|V(G)|$ is minimum among all counterexamples. We use $1_a, 1_b, 2_a, 2_b$ to denote the colors of a $(1^2, 2^2)$-packing edge-coloring.

Since $G$ is $3$-irregular subcubic with no loop, three parallel edges do not exist. Suppose $e_1$ and $e_2$ are two edges with the same endpoints $u_1,u_2$. If both of $u_1$ and $u_2$ have degree two, then we are done. If $u_1$ has degree three and $u_2$ has degree two, then $G' = G - \{u_1, u_2\}$ has a $(1^2,2^2)$-packing edge-coloring $f$ since $G$ is a minimum counterexample. Let $N(u_1) = \{v_1, u_2\}$. Since $G$ is $3$-irregular, $v_1$ has degree at most two and has no parallel edges. If $v_1$ is a $1$-vertex, then we are done. Thus, $v_1$ is a $2$-vertex and let $N(v_1) = \{u_1, v_2\}$. We can extend $f$ to $G$ to obtain a contradiction. If $f(v_1v_2) \in \{1_a, 1_b\}$, say $f(v_1v_2) = 1_a$, then we color $u_1v_1, e_1, e_2$ with $1_b, 2_a, 2_b$. Otherwise, $f(v_1v_2) \in \{2_a, 2_b\}$, say $f(v_1v_2) = 2_a$. Then we color $u_1v_1, e_1, e_2$ with $1_a, 1_b, 2_b$.

We now prove the main statement: Every $3$-irregular subcubic simple graph has a $(1^2, 2^2)$-packing edge-coloring. Let $G$ be a $3$-irregular subcubic simple graph. We may assume $G$ is connected since otherwise we can provide a $(1^2, 2^2)$-packing edge-coloring of $G$ by coloring each component of $G$. We may also assume $\delta(G) \ge 2$ since every $3$-irregular subcubic simple graph is a subgraph of a $3$-irregular subcubic simple graph with minimum degree two. We use $1_a, 1_b, 2_a, 2_b$ to denote the colors of a $(1^2, 2^2)$-packing edge-coloring.

Take two disjoint matchings $M_1, M_2$ of $G$ such that 
\begin{equation}\label{condition-1}
|M_1 \cup M_2| \text{ is maximum among all combinations of two disjoint matchings.} 
\end{equation}
Let $G_{M_1,M_2}$ be the graph induced by the edges of $G-M_1-M_2$. We take $M_1,M_2$ subject to (1) and further satisfy the following condition: 
\begin{equation}\label{condition-2}
\text{The graph } G_{M_1,M_2} \text{ has minimum number of connected components.}  
\end{equation}

\begin{lemma}\label{p2p3}
Each component of $G_{M_1,M_2}$ can either be a $P_2$ or $P_3$. Furthermore, if a component is a $P_3$, then the middle vertex of this $P_3$ has degree two in $G$.   
\end{lemma}

\begin{proof}
We first show each component of $G_{M_1,M_2}$ must have maximum degree at most two. Suppose not, i.e., there is a vertex $u$ of degree three. Let $N(u) = \{u_1, u_2, u_3\}$ and $N(u_i) = \{u, v_i\}$, where $i \in [3]$. If $u_1v_1 \in G-M_1-M_2$, then we add $uu_1$ to $M_1$ to obtain a contradiction with Condition~\eqref{condition-1}. We may assume $u_1v_1 \in M_1$. However, we can add $uu_1$ to $M_2$, which is again a contradiction with Condition~\eqref{condition-1}.

Suppose there is a $P_4$, $u_1u_2u_3u_4$, in a component of $G_{M_1,M_2}$. One of $u_2$ and $u_3$ must be a $3$-vertex since otherwise we can add $u_2u_3$ to $M_1$, which contradicts Condition~\eqref{condition-1}. Thus, we may assume $u_2$ is a $3$-vertex and both $u_1,u_3$ are $2$-vertices since $G$ is a $3$-irregular subcubic simple graph. Say $N(u_2) = \{u_1,u_3,u_5\}$ and $u_2u_5 \in M_1$. Then we add $u_2u_3$ to $M_2$, which again contradicts Condition~\eqref{condition-1}.

Suppose there is a $P_3$, $u_1u_2u_3$, and the middle vertex $u_2$ has degree three. Let $N(u_2) = \{u_1, u_3, v_2\}$. Since $G$ is $3$-irregular, $u_1,u_3$ are both $2$-vertices. Let $N(u_1) = \{u_2, v_1\}$ and $N(u_3) = \{u_2, v_3\}$. Since $\Delta(G_{M_1,M_2}) \le 2$, $u_2v_2 \in M_1 \cup M_2$. Say $u_2v_2 \in M_1$. Then $u_1v_1 \in M_2$ since otherwise we add $u_1u_2$ to $M_2$ and will obtain a contradiction with Condition~\eqref{condition-1}. Similarly, $u_3v_3 \in M_2$. Furthermore, $u_1v_1$ and $u_2v_2$ must be in the same component of $G[M_1 \cup M_2]$ since otherwise, say $u_1v_1$ is in component $G_1$ and $u_2v_2$ is in component $G_2$ of $G[M_1 \cup M_2]$, we switch the $M_1$ and $M_2$ edges in $G_1$ and add $u_1u_2$ to $M_2$, which is a contradiction with Condition~\eqref{condition-1}. Similarly, $u_2v_2$ and $u_3v_3$ must be in the same component of $G[M_1 \cup M_2]$. However, $u_1,u_2,u_3$ are three leaves in a component of $G[M_1 \cup M_2]$ and it contradicts the fact that $\Delta(G_{M_1,M_2}) \le 2$.
\end{proof}

Let $H$ be the graph with $V(H) = E(G)-M_1-M_2$ and $E(H) = \{e_1e_2: e_1, e_2 \in V(H) \text{ and } dist_G(e_1, e_2) \le 2 \}$. We first show two $P_3$s in $G-M_1-M_2$ cannot be in the same component of $H$.

\begin{lemma}\label{notwop3-1}
Two $P_3$s in $G-M_1-M_2$ cannot be joined by an edge in $M_1 \cup M_2$.
\end{lemma}

\begin{proof}
Suppose not, i.e., $u_1u_2u_3$ and $v_1v_2v_3$ are two $P_3$s in $G-M_1-M_2$ such that $u_3v_1 \in M_1 \cup M_2$. Say $u_3v_1 \in M_1$. By Lemma~\ref{p2p3}, $u_2$ and $v_2$ are $2$-vertices. If $u_3$ is a $2$-vertex, then we add $u_2u_3$ to $M_2$ to obtain a contradiction with Condition~\ref{condition-1}. Therefore, $u_3$ is a $3$-vertex. Since $G$ is $3$-irregular, $v_1$ is a $2$-vertex. We add $v_1v_2$ to $M_2$ to obtain a contradiction with Condition~\ref{condition-1}.
\end{proof}

\begin{lemma}\label{notwop3-2}
If $u_1u_2u_3$ and $v_1v_2v_3$ are two $P_3$s, then $u_1u_2,u_2u_3$ and $v_1v_2, v_2v_3$ are not in the same connected component of $H$.
\end{lemma}

\begin{proof}
Suppose not, i.e., $u_1u_2,u_2u_3$ and $v_1v_2, v_2v_3$ are in the same component. By Lemma~\ref{notwop3-1}, $u_1u_2u_3$ and $v_1v_2v_3$ cannot be joined by an edge in $M_1 \cup M_2$. Since $u_1u_2,u_2u_3$ and $v_1v_2, v_2v_3$ are in the same component in $H$, $u_1u_2u_3$ and $v_1v_2v_3$ must be connected by a series of $P_2$s in $G-M_1-M_2$ via edges in $M_1 \cup M_2$. We prove the statement ``there cannot be two $P_3$s connected by $k$ $P_2$s'' by induction, where $k \ge 0$ and is an integer.

The case when $k = 0$ are covered by Lemma~\ref{notwop3-1}. The remaining base case is when $k = 1$, i.e., $u_1u_2u_3$ and $v_1v_2v_3$ are connected by a $P_2$, say $w_1w_2$, with $u_3w_1, w_2v_1  \in M_1 \cup M_2$. Say $u_3w_1 \in M_1$. By Lemma~\ref{p2p3}, $u_2$ and $v_2$ are both $2$-vertices. If $u_3$ is a $2$-vertex, then we add $u_2u_3$ to $M_2$ and obtain a contradiction with Condition~\ref{condition-1}. Therefore, $u_3$ must be a $3$-vertex. Similarly, $v_1$ must be a $3$-vertex. Let $N(u_3) = \{u_2, w_1, u_3'\}$ and $N(v_1) = \{w_2, v_2, v_1'\}$. We know $u_3u_3' \in M_2$. Since $G$ is $3$-irregular, both $w_1$ and $w_2$ are $2$-vertices. Therefore, $w_2v_1 \in M_2$ since otherwise we can add $w_1w_2$ to $M_2$ to obtain a contradiction with Condition~\ref{condition-1}. We also know $v_1v_1' \in M_1$. Then we can delete $u_3w_1$ from $M_1$, $w_2v_1$ from $M_2$, and add $u_2u_3$ to $M_1$, $v_1v_2$ to $M_2$, and add $w_1w_2$ to $M_1$. This is a contradiction with Condition~\ref{condition-1} since $M_1 \cup M_2$ becomes larger.

We assume the statement is already proved for $k \le \ell -1$ and now prove the case when $k = \ell$. Suppose $u_1u_2u_3$ and $v_1v_2v_3$ are connected by $\ell$ $P_2$s, where $\ell \ge 2$. Let the first (starting from $u_1u_2u_3$) $P_2$ connecting $u_1u_2u_3$ and $v_1v_2v_3$ be $w_1w_2$ and the last $P_2$ be $w_3w_4$, with $u_3w_1, w_4v_1 \in M_1 \cup M_2$. By a similar argument, we know $u_3,v_1$ are $3$-vertices and $u_2,w_1,w_4,v_2$ are $2$-vertices. Let $u_3w_1 \in M_x$ and $w_4v_1 \in M_y$, where $x,y \in [2]$. We delete $u_3w_1$ from $M_x$, $w_4v_1$ from $M_y$, and add $u_2u_3$ to $M_x$, $v_1v_2$ to $M_y$. Then we obtained two $P_3$s $u_3w_1w_2$ and $w_3w_4v_1$ connected by $2$ fewer $P_2$s, which contradicts with the inductive hypothesis.
\end{proof}

\begin{lemma}\label{nop3inoddcycle}
If $u_1u_2u_3$ is a $P_3$ in $G-M_1-M_2$, then $u_1u_2, u_2u_3$ cannot be together in a cycle of $H$.
\end{lemma}

\begin{proof}
Suppose not, i.e., $u_1u_2u_3$ is a $P_3$ in $G-M_1-M_2$ and $u_1u_2, u_2u_3$ are in the same cycle of $H$. By Lemma~\ref{notwop3-2}, there is no other $P_3$s in the cycle and let the corresponding cycle in $G$ be $u_1u_2u_3u_4 \cdots u_{2k+1}u_1$.  We know $u_{2i}u_{2i+1} \in G-M_1-M_2$, where $2 \le i \le k$, $u_{2i-1}u_{2i} \in M_1 \cup M_2$, where $2 \le i \le k$, and $u_{2k+1}u_1 \in M_1 \cup M_2$. By Lemma~\ref{p2p3}, $u_1$ and $u_3$ are $3$-vertices. Since $G$ is $3$-irregular, there must be at least one $i$, where $2 \le i \le k$, such that both $u_{2i}$ and $u_{2i+1}$ are $2$-vertices. Then we may assume $u_{2i-1}u_{2i} \in M_x$ and $u_{2i+1}u_{2i+2} \in M_y$, where $x,y \in [2]$ and $x \neq y$. Then we delete $u_{2i+1}u_{2i+2}$ from $M_y$ and add $u_{2i}u_{2i+1}$ to $M_y$. We either obtain two $P_3$s in $G-M_1-M_2$ connecting by a series of $P_2$s, which contradicts Lemma~\ref{notwop3-2}, or we obtain a $P_4$ in $G-M_1-M_2$, which contradicts Lemma~\ref{p2p3}.
\end{proof}

\begin{lemma}\label{nooddcycle}
There is no odd cycle in $H$.    
\end{lemma}

\begin{proof}
Suppose not, i.e., there is an odd cycle $e_1e_2 \cdots e_k$ in $H$. Let $e_i = u_{2i-1}u_{2i}$, where $1 \le i \le k$. Then there is a corresponding cycle in $G$: $u_1u_2 \cdots u_{2k}u_1$. By Lemma~\ref{nop3inoddcycle}, there is no $P_3$ in $G-M_1-M_2$ in the cycle. Thus, we may assume $e_i = u_{2i-1}u_{2i} \in G-M_1-M_2$ and $u_{2i}u_{2i+1} \in M_1 \cup M_2$, where $1 \le i \le k$ and $u_{2k+1}:=u_1$. If there is an $i$, where $1 \le i \le k$, such that both $u_{2i-1}$ and $u_{2i}$ are $2$-vertices and say $u_{2i-2}u_{2i-1} \in M_1$, then $u_{2_i}u_{2i+1} \in M_2$ since otherwise we can add $u_{2i-1}u_{2i}$ to $M_2$ and it contradicts Condition~\ref{condition-1}. Then we delete $u_{2i}u_{2i+1}$ from $M_2$ and add $u_{2i-1}u_{2i}$ to $M_2$ (it will be refered as the ``switching argument'' in the context). It contradicts Condition~\ref{condition-2} since $u_{2i}u_{2i+1}, u_{2i+1}u_{2i+2}$ froms a $P_3$ in $G-M_1-M_2$ and the number of connected components in $G-M_1-M_2$ dropped by one.

We may assume $u_1$ is a $3$-vertex. Since $G$ is $3$-irregular and one of each pair $(u_{2i-1}, u_{2i})$,  $1 \le i \le k$, must be a $3$-vertex, we know each $u_{2i-1}$ is a $3$-vertex, where $1 \le i \le k$. Let $N(u_{2i-1}) = \{u_{2i-2}, u_{2i}, v_{2i-1}\}$, $1 \le i \le k$. Without loss of generality, we assume $u_{2k}u_1 \in M_1$. By a similar switching argument, we know each $u_{2i}u_{2i+1} \in M_1$ and each $u_{2i-1}v_{2i-1} \in M_2$, where $1 \le i \le k$. If we can change $u_{2k}u_1, u_1v_1$ to be in $M_2, M_1$, then we can apply a similar switching argument to obtain a contradiction with Condition~\ref{condition-2}. Therefore, $u_{2k}, u_1, v_1$ and $u_2, u_3, v_3$ are in the same connected component of $G[M_1 \cup M_2]$. Similarly, $u_2, u_3, v_3$ and $u_4, u_5, v_5$ are in the same connected component of $G[M_1 \cup M_2]$. Then each of $u_{2k}, u_2, u_4$ is a $1$-vertex in $G[M_1 \cup M_2]$ and they are in the same connected component of $G[M_1 \cup M_2]$, which is a contradiction with the fact that $\Delta(G[M_1 \cup M_2]) \le 2$.
\end{proof}

\textbf{Final Proof of Theorem~\ref{theorem-1}:} By Lemma~\ref{nooddcycle}, $H$ is a bipartite graph and thus can be colored by two colors $a,b$. The colors $a,b$ for vertices in $H$ can be applied to its corresponding edges in $G-M_1-M_2$. If we rename the color $a$ with $2_a$ and $b$ with $2_b$, color each edge in $M_1$ with $1_a$ and each edge in $M_2$ with $1_b$, then it will result in a $(1^2,2^2)$-packing edge-coloring of $G$ using the colors $1_a,1_b,2_a,2_b$.  \hfill \qed


\section{Proof of Theorem~\ref{theorem-2}}
We use $1_a, 1_b, 2_a, 2_b$ to denote the four colors of a $(1^2,2^2)$-packing edge-coloring. A {\em $k$-thread} is a path on $k$ vertices of degree two. We call a $(1^2,2^2)$-packing edge-coloring {\em good coloring} if it further satisfies:

Condition (I): if $e_1, e_2$ are colored with $2_a, 2_b$, then they cannot share a common endpoint;

Condition (II): if $e_1 = u_1u_2, e_2=v_1v_2$ do not share a common endpoint and are colored with $2_a, 2_b$, then there is no vertex $w$ such that $u_iwv_j$ is a path in $G$, where $1 \le i,j \le 2$.

Suppose to the contrary that there is a subcubic graph $G$ with maximum average degree less than $\frac{20}{9}$, but it has no good coloring. We take such a $G$ with minimum $|V(G)|+|E(G)|$. We may assume $G$ is connected, since otherwise we can apply the same argument to each component of $G$. We first show $G$ has no $1$-vertex and then show a few reducible configurations.

\begin{lemma}\label{mindegree}
$\delta(G) \ge 2$.
\end{lemma}

\begin{proof}
Suppose $G$ has a $1$-vertex $u$. Let $N(u) = \{u_1\}$. We know $u_1$ is a $3$-vertex, since otherwise, say $u_1$ is a $2$-vertex (if $u_1$ is a $1$-vertex, then the entire graph is a $K_2$) with neighbours $u,u_2$, we delete $u$ from $G$ to obtain a subcubic graph $G'$ with $\mad(G) < \frac{20}{9}$. By the minimality of $G$, $G'$ has a good coloring $f$. We color $uu_1$ with a color in $\{1_a, 1_b\} \setminus \{f(u_1u_2)\}$ to obtain a good coloring, which is a contradiction. Let $N(u_1) = \{u, u_2, u_3\}$. Note $u_2,u_3$ cannot be both $1$-vertices, since otherwise $G$ only has three edges. We show none of $u_2, u_3$ is a $1$-vertex. Suppose not, say $u_2$ is a $1$-vertex. We delete $u,u_2$ to obtain a subcubic graph $G'$ with $\mad(G') < \frac{20}{9}$. By the minimality of $G$, $G'$ has a good coloring $f$. If $f(u_1u_3) \in \{2_a, 2_b\}$, then we color $uu_1, u_1u_2$ with $1_a, 1_b$ to obtain a good coloring. Otherwise, say $f(u_1u_3) = 1_a$. In case $u_3$ is a $3$-vertex, then, by Condition (I), $u_3$ sees exactly one of $2_a, 2_b$, say $2_a$. We color $uu_1, u_1u_2$ with $2_b, 1_b$. Since $u_3$ sees $2_a$, the $2_b$ used on $uu_1$ satisfies Condition (I) and (II), and it is a good coloring. In case $u_3$ is a $2$-vertex, say $N(u_3) = \{u_1, u_4\}$. We may assume $f(u_3u_4) = 1_b$, since if $f(u_3u_4) \in \{2_a, 2_b\}$ then we obtain a good coloring similarly to the case when $u_3$ is a $3$-vertex. By Condition (I), $u_4$ sees at most one of $\{2_a, 2_b\}$, say $u_4$ does not see $2_b$. We color $uu_1, u_1u_2$ with $2_a, 1_b$. This is a good coloring and we are done. Therefore, we assume none of $u_2,u_3$ are $1$-vertices. We delete $u$ to obtain a subcubic graph $G'$ with $\mad(G') < \frac{20}{9}$. By the minimality of $G$, $G'$ has a good coloring $f$. By Condition (I), we may assume $\{f(u_1u_2), f(u_1u_3)\} = \{1_a, 2_a\}$ or $\{1_a, 1_b\}$. In the former case, we color $uu_1$ with $1_b$ to obtain a good coloring. In the latter case, we may assume $f(u_1u_2) = 1_a$ and $f(u_1u_3) = 1_b$. Note that it is possible that $u_2u_3 \in E(G)$ but it will not influence our proof. 

\textbf{Case 1:} Both $u_2, u_3$ are $2$-vertices. Note that $u_2u_3 \notin E(G)$ since otherwise the graph only has $4$ edges. Let $N(u_2) = \{u_1,u_4\}$ and $N(u_3) = \{u_1, u_5\}$. By Condition (II), $\{f(u_2u_4), f(u_3u_5)\} \neq \{2_a, 2_b\}$. By symmetry, $f(u_2u_4), f(u_3u_5) = 2_a, 2_a$ or $1_b, 2_a$ or $1_b, 1_a$. In the former case, we color $uu_1$ with $2_b$ to obtain a good coloring. In the middle case, we color $uu_1$ with $2_b$ to obtain a good coloring unless $u_4$ sees $2_a$. However, we can recolor $u_1u_2$ with $2_b$ and color $uu_1$ with $1_a$. This is a good coloring since $f(u_3u_5) = 2_a$ and $u_4$ sees $2_a$. In the latter case, we can color $uu_1$ with $2_a$ ($2_b$) unless $u_4$ or $u_5$ sees $2_b$ ($2_a$). Therefore, we may assume $u_4$ sees $2_a$ and $u_5$ sees $2_b$. We recolor $u_1u_2$ with $2_b$ and color $uu_1$ with $1_a$. Since $u_4$ sees $2_a$ and $u_5$ sees $2_b$, it is a good coloring.

\textbf{Case 2:} One of $u_2, u_3$ is a $2$-vertex and the other is a $3$-vertex. Say, $u_2$ is a $3$-vertex and $u_3$ is a $2$-vertex. Let $N(u_2) = \{u_1, u_4, u_5\}$ and $N(u_3) = \{u_1, u_6\}$. By Condition (I), $u_2$ sees exactly one $2$-color, say $f(u_2u_4) = 2_a$ and $f(u_2u_5) = 1_b$. By Condition (II), $f(u_3u_6) \neq 2_b$. If $f(u_3u_6) = 2_a$, then we color $uu_1$ with $2_b$. Since $f(u_3u_6) = f(u_2u_4) = 2_a$, this coloring satisfies Condition (I) and (II). Otherwise, $f(u_3u_6) = 1_a$. If $u_6$ does not see $2_a$, then we color $uu_1$ with $2_b$. Otherwise, $u_6$ sees $2_a$, and thus we can recolor $u_1u_3$ with $2_b$ and color $uu_1$ with $1_b$ to obtain a good coloring. 

\textbf{Case 3:} Both $u_2,u_3$ are $3$-vertices. By Condition (I) and both $u_2,u_3$ are $3$-vertices, each of $u_2, u_3$ sees exactly one $2$-color. By Condition (II), we may assume both $u_2, u_3$ see $2_a$. We color $uu_1$ with $2_b$ and it satisfies Condition (I) and (II). Therefore, it is a good coloring and we obtain a contradiction.
\end{proof}

\begin{lemma}\label{4thread}
There is no $4$-thread.    
\end{lemma}

\begin{proof}
Suppose not, say $u_1v_1v_2v_3v_4u_2$ is a $4$-thread. We delete $v_2, v_3$ to obtain a subcubic graph $G'$ with $\mad(G') < \frac{20}{9}$. By the minimality of $G$, $G'$ has a good coloring $f$. By symmetry, $f(u_1v_1), f(u_2v_4) = 1_a, 1_b$ or $1_a, 1_a$ or $1_a, 2_a$ or $2_a, 2_b$ or $2_a, 2_a$. In the last four cases, we obtain a good coloring of $G$ by coloring $v_1v_2, v_2v_3, v_3v_4$ with a $1_a$-$1_b$ chain. In the first case, we color $v_1v_2, v_2v_3, v_3v_4$ with $2_a, 1_b, 1_a$ to obtain a good coloring unless $u_1$ sees $2_a$. However, we color $v_1v_2, v_2v_3, v_3v_4$ with $2_b, 1_b, 1_a$ to obtain a good coloring.
\end{proof}

\begin{lemma}\label{3thread}
If a $3$-vertex $u$ is adjacent to two $3$-threads, then it cannot adjacent to another $2$-thread.    
\end{lemma}

\begin{proof}
Suppose not, say $u$ is adjacent to two $3$-threads $uv_1v_2v_3v$ and $uw_1w_2w_3w$, where $v,w$ are $3$-vertices, and a $2$-thread $uu_1u_2u'$. We delete $v_2$ to obtain a subcubic graph $G'$ with $\mad(G') < \frac{20}{9}$. By the minimality of $G$, $G'$ has a good coloring $f$. By symmetry,  $f(uv_1), f(vv_3) = 1_a, 1_b$ or $1_a, 2_a$ or $2_a, 2_a$ or $1_a, 1_a$ or  $2_a, 2_b$. In the first three cases, we obtain a good coloring by coloring $v_1v_2, v_2v_3$ with a $1_a,1_b$ or $1_b, 1_a$. In the fourth case, say $u$ sees $2_a$. We can color $v_1v_2, v_2v_3$ with $2_b, 1_b$ unless $v$ sees $2_a$. If $f(uu_1) = 2_a$, then $uw_1, w_1w_2, w_2w_3, w_3w$ must have colors $1_b, 1_a, 1_b, 1_a$, since otherwise we can switch the $1_a$-$1_b$ chain starting from $uv_1$ so that $uv_1$ is colored with $1_b$. We obtain a good coloring by using $1_a, 1_b$ at $v_1v_2, v_2v_3$. By Condition (I), $f(u_1u_2) \neq 2_b$. We further claim $f(u_1u_2) = 1_a$, since otherwise we recolor $uu_1, uv_1$ with $1_a, 2_b$, and color $v_1v_2, v_2v_3$ with $1_a, 1_b$. We also know $f(u_2u') = 1_b$ since otherwise we recolor $u_1u_2$ with $1_b$, which is a contradiction. Thus, we recolor $uu_1, uw_1$ with $1_b, 2_b$, and color $v_1v_2, v_2v_3$ with $2_a, 1_b$ to obtain a good coloring. This is a contradiction and hence we can assume $f(uw_1) = 2_a$. Similarly to the proof for $f(uu_1) = 2_a$, we know $uu_1,u_1u_2, u_2u'$ have colors $1_b, 1_a, 1_b$. Furthermore, $w_1w_2, w_2w_3, w_3w$ must be $1_a, 1_b, 1_a$ since otherwise we can recolor so that $w_1w_2$ is colored with $1_b$, recolor $uw_1, uv_1$ with $1_a, 2_a$, and color $v_1v_2, v_2v_3$ with $1_a, 1_b$ to obtain a good coloring. However, we recolor $uw_1$ with $2_b$ and color $v_1v_2, v_2v_3$ with $2_a, 1_b$ to obtain a good coloring.

In the last case, we must recolor one of $uv_1$ and $vv_3$ since it does not satisfy Condition (II). By Condition (I), we may assume $f(uu_1) = 1_a$ and $f(uw_1) = 1_b$. If $f(w_1w_2) = 2_b$, then, $f(ww_3) \neq 2_b$, and by Condition (I), $f(w_2w_3) \neq 2_a$. We recolor $uw_1$ and $uv_1$ with $2_a, 1_b$, recolor $w_1w_2$ with an available $1$-color, and color $v_1v_2, v_2v_3$ with $1_a, 1_b$. Since $f(ww_3) \neq 2_b$, this is a good coloring. Therefore, we may assume $f(w_1w_2) = 1_a$. By Condition (II), we know $f(w_2w_3) \neq 2_b$. We further claim $f(w_2w_3) = 1_b$. Suppose not, $f(w_2w_3) = 2_a$, and we must have $f(ww_3) = 1_b$, since otherwise we recolor $w_2w_3$ with $1_b$, recolor $uw_1, uv_1$ with $2_a, 1_b$, and color $v_1v_2, v_2v_3$ with $1_a, 1_b$ to obtain a good coloring. However, we recolor $uv_1, uw_1, w_1w_2, w_2w_3$ with $1_b, 2_a, 1_b, 1_a$, and color $v_1v_2, v_2v_3$ with $1_a, 1_b$ to obtain a good coloring. Therefore, $f(w_2w_3) = 1_b$. We can recolor $uv_1, uw_1$ with $1_b, 2_a$ unless $f(ww_3) = 2_b$. Then we claim $f(u_1u_2) = 2_b$, since otherwise, $f(u_1u_2) = 1_b$, and we recolor $uv_1, uw_1$ with $1_b, 2_b$ and color $v_1v_2, v_2v_3$ with $1_a, 1_b$ to obtain a good coloring unless $f(u_2u') = 2_a$. However, we recolor $uv_1, uu_1$ with $1_a, 2_b$ and color $v_1v_2, v_2v_3$ with $1_b, 1_a$ to obtain a good coloring. By Condition (I), $f(u'u_2) \in \{1_a, 1_b\}$. We recolor $u_1u_2, uu_1, uw_1, w_1w_2, w_2w_3$ with an available $1_a$-$1_b$ chain, recolor $uv_1$ with $2_b$, and color $v_1v_2, v_2v_3$ with $1_a, 1_b$. This is a good coloring and we obtain a contradiction.
\end{proof}

We finish the proof by a discharging argument. Let the initial charge of each vertex be $\ch(v) = d(v) - \frac{20}{9}$. Since $\mad(G) < \frac{20}{9}$, 
\begin{equation}\label{equation1}
\sum\limits_{v \in V(G)} \ch(v) = \sum\limits_{v \in V(G)} (d(v) - \frac{20}{9}) = 2 |E(G)| - \frac{20}{9} |V(G)| < 0.    
\end{equation}
We redistribute the charges according to the following rule:

\textbf{Rule 1:} Every $3$-vertex gives each $2$-vertex $\frac{1}{9}$ if they can be connected by a path of $2$-vertices.

By Lemma~\ref{mindegree}, we only need to consider $2$-vertices and $3$-vertices. By Lemma~\ref{4thread} and~\ref{3thread}, each $3$-vertex can be connected with at most seven $2$-vertices via a path of $2$-vertices. Therefore, each $3$-vertex $u$ has final charge $\ch^*(u) \ge 3 - 7 \cdot \frac{1}{9} - \frac{20}{9} = 0$. On the other hand, each $2$-vertex connects with exactly two $3$-vertices via a path of $2$-vertices. Therefore, each $2$-vertex $v$ has final charge $\ch^*(v) \ge 2+ 2 \cdot \frac{1}{9} - \frac{20}{9} = 0$. However, by equation~\eqref{equation1},
$$ 0 \le \sum\limits_{v \in V(G)} \ch^*(v) = \sum\limits_{v \in V(G)} \ch(v) < 0,$$
which is a contradiction. \hfill \qed

\section{Open Question}

We end this paper by proposing the following open question:

\begin{oq}
What is the minimum positive integer $g$ such that every subcubic planar graph with girth at least $g$ is $(1^2,2^2)$-packing edge-colorable?
\end{oq}

\end{document}